\documentclass[reqno,12pt]{amsart}
\usepackage{amsmath,amssymb,latexsym}	
\usepackage{color}
\numberwithin{table}{section}

\newcommand{\Sp}{\mathrm{Sp}}

\newcommand{\GL}{\mathrm{GL}}

\newcommand{\R}{\mathbb{R}}
\newcommand{\Q}{\mathbb{Q}}
\newcommand{\Z}{\mathbb{Z}}

\newcommand{\C}{\mathbb{C}}
\newcommand{\e}{\epsilon}

\newtheorem{theorem}{Theorem}[section]
\newtheorem{lemma}[theorem]{Lemma}
\newtheorem{proposition}[theorem]{Proposition}

\theoremstyle{definition}

\newtheorem{rmk}[theorem]{Remark}
\newtheorem{question}[theorem]{Question}

\author{Sandip  Singh}
\address{{Max Planck Institute for Mathematics, Vivatsgasse 7, 53111 Bonn, Germany}}
\email{{{singhs@mpim-bonn.mpg.de; sonutifr@gmail.com}}}  

\begin{document} \title[Orthogonal Hypergeometric Groups]{Orthogonal Hypergeometric Groups with a Maximally Unipotent Monodromy}

\vskip 5mm
\begin{abstract} Similar to the symplectic cases, there is a family of $14$ orthogonal hypergeometric groups with a maximally unipotent monodromy (cf. Table \ref{table:calabiyau}). We show that $2$ of the $14$ orthogonal  hypergeometric groups associated to the pairs of parameters $(0, 0, 0, 0, 0)$, $(\frac{1}{6}, \frac{1}{6}, \frac{5}{6}, \frac{5}{6}, \frac{1}{2})$;  and $(0, 0, 0, 0, 0)$, $(\frac{1}{4}, \frac{1}{4}, \frac{3}{4}, \frac{3}{4}, \frac{1}{2})$ are arithmetic. 

We also give a table (cf. Table \ref{table:calabiyau2}) which lists the quadratic forms $\mathrm{Q}$ preserved by these 14 hypergeometric groups, and their two linearly independent $\mathrm{Q}$- orthogonal isotropic vectors in $\Q^5$; it shows in particular that the orthogonal groups of these quadratic forms have $\Q$- rank two.
\end{abstract}
\maketitle

\section{Introduction}
To explain the results of this paper, we first recall the definition of {\it hypergeometric groups}. For, we denote $\theta=z\frac{d}{dz}$ and write the differential operator:
\begin{align*}
D(\alpha;\beta)&:=(\theta+\beta_1-1)\cdots(\theta+\beta_n-1)-z(\theta+\alpha_1)\cdots(\theta+\alpha_n)
\end{align*}
for $\alpha=(\alpha_1,\ldots,\alpha_n)$, $\beta=(\beta_1,\ldots,\beta_n)\in\C^n$; and consider the hypergeometric differential equation
\begin{eqnarray*}\label{introdifferentialequation}
D(\alpha;\beta)w=0
\end{eqnarray*}
on $\mathbb{P}^1(\C)$ with {\it regular singularities} at the points $0, 1$ and $\infty$, and regular elsewhere.

The fundamental group $\pi_1$ of $\mathbb{P}^1(\C)\backslash\{0,1,\infty\}$ acts on the solution space of the differential equation $D(\alpha;\beta)w=0$, and we get a representation $\rho$ of $\pi_1$ inside $\GL_n(\C)$: we call the representation $\rho$ the {\it monodromy} representation,  and the subgroup $\rho(\pi_1)$ of $\GL_n(\C)$ {\it monodromy group} of the hypergeometric equation $D(\alpha;\beta)w=0$; we also call the monodromy group $\rho(\pi_1)$ {\it hypergeometric group}. 

Note that $\rho(\pi_1)$  is generated by the monodromy matrices $\rho(h_0)$, $\rho(h_1)$, $\rho(h_\infty)$, where $h_0, h_1, h_\infty$ (loops around $0, 1, \infty$ resp.) are the generators of $\pi_1$ with a single relation $h_\infty h_1 h_0=1$. 

The hypergeometric groups $\rho(\pi_1)$ are classified by a theorem of Levelt (\cite{Le}; cf. \cite[Theorem 3.5]{BH}): 
if $\alpha_1,\alpha_2,\ldots,\alpha_n$, $\beta_1,\beta_2,\ldots,$ $\beta_n\in\C$ such that $\alpha_j-\beta_k\not\in\Z$, for all $j,k=1,2,\ldots,n$,  then the hypergeometric group $\rho(\pi_1)$ is  (up to conjugation in $\GL_n(\C)$)  a subgroup of $\GL_n(\C)$ generated by the {\it companion matrices} $A$ and $B$ of  $$f(X)=\prod_{j=1}^{n}(X-{\rm{e}^{2\pi i\alpha_j}})\quad\mbox{ and }\quad g(X)=\prod_{j=1}^{n}(X-{\rm{e}^{2\pi i\beta_j}})$$ resp., and the monodromy $\rho$ is defined by  $h_\infty\mapsto A$, $h_0\mapsto B^{-1}$, $h_1\mapsto A^{-1}B$. 

We now denote the hypergeometric group $\rho(\pi_1)$ by $\Gamma(f,g)$ (which is a subgroup of $\GL_n(\C)$ generated by the companion matrices of $f,g$) and consider the cases where $\alpha, \beta\in\Q^n$ (that is, the roots of the polynomials $f,g$ are the roots of unity) and the coefficients of $f,g$ are integers (for example, one can take $f,g$ as product of cyclotomic polynomials); in these cases, $\Gamma(f,g)\subset\GL_n(\Z)$. Note that the condition ``$\alpha_j-\beta_k\not\in\Z$, for all $j,k=1,2,\ldots,n$'' means that $f,g$ do not have any common root in $\C$. We also assume that $f,g$ form a primitive pair \cite[Definition 5.1]{BH}. 
 
 With the assumptions made in last paragraph, the Zariski closures $\mathrm{G}$ of the hypergeometric groups $\Gamma(f,g)$ are completely determined by Beukers and Heckman in \cite[Theorem 6.5]{BH}: if $n$ is even and $f(0)=g(0)=1$, then the hypergeometric group $\Gamma(f,g)$ preserves a non-degenerate integral {\it symplectic} form $\Omega$ on $\Z^n$ and $\Gamma(f,g)\subset\Sp_\Omega(\C)$ is Zariski dense, that is, $\mathrm{G}=\Sp_\Omega$; in other cases, if $\Gamma(f,g)$ is infinite and $\frac{f(0)}{g(0)}=-1$, then $\Gamma(f,g)$ preserves a non-degenerate integral {\it quadratic} form $\mathrm{Q}$ on $\Z^n$ and $\Gamma(f,g)\subset\mathrm{O}_\mathrm{Q}(\C)$ is Zariski dense, that is, $\mathrm{G}=\mathrm{O}_\mathrm{Q}$.
 
 Note that a hypergeometric group $\Gamma(f,g)$ is called {\it arithmetic}, if it is of {\it finite} index in $\mathrm{G}(\Z)$; and {\it thin}, if it has {\it infinite} index in $\mathrm{G}(\Z)$ \cite{S}. In \cite{S}, Sarnak has asked a question to determine the cases where $\Gamma(f,g)$ is arithmetic or thin. There have been some progress to answer the question of Sarnak. 
 
 For the symplectic cases: infinitely many arithmetic $\Gamma(f,g)$ in $\Sp_n$ (for any even $n$) are given by the author and Venkataramana in \cite{SV}; four other cases of arithmetic $\Gamma(f,g)$ in $\Sp_4$ are given by the author in \cite{SS}; seven thin $\Gamma(f,g)$ in $\Sp_4$ are given by Brav and Thomas in \cite{BT}; and in \cite{HvS}, Hofmann and van Straten  have determined the index of $\Gamma(f,g)$ in $\Sp_4(\Z)$ for some of the arithmetic cases of \cite{SS} and \cite{SV}.
 
 For the orthogonal cases: when the quadratic form $\mathrm{Q}$ has signature $(n-1,1)$, Fuchs, Meiri and Sarnak give infinitely many thin $\Gamma(f,g)$ in \cite{FMS}; when the quadratic form $\mathrm{Q}$ has signature $(p,q)$ with $p,q\geq2$, infinitely many arithmetic $\Gamma(f,g)$ are given by Venkataramana in a very recent paper \cite{V} (these are also the first examples of higher rank arithmetic orthogonal hypergeometric groups); and an example of thin $\Gamma(f,g)$ in $\mathrm{O}(2,2)$ is given by Fuchs in \cite{EF}.
 
 We now take particular examples where $f=(X-1)^4$ ({{that is}}, the local monodromy is maximally unipotent at $\infty$ and $\alpha=(0,0,0,0)$). Let $g(X)$ be  the product of cyclotomic polynomials such that $g(0)=1$, $g(1)\neq 0$,  and $f,g$ form a primitive pair. There are precisely $14$ such examples, which have been listed in \cite{AvEvSZ}, \cite{YYCE}, \cite{DoMo}, \cite{SS}, and \cite{SV}. Then, it turns out (for a reference, see \cite{AvEvSZ}, \cite{YYCE}, and \cite{DoMo}) that, this monodromy is same as the monodromy of $\pi_1\left(\mathbb{P}^1(\C)\backslash\{0,1,\infty\}\right)$ on certain {{subspaces}} of $\mathrm{H}^3$ of the fibre of a family $\{Y_t : t\in\mathbb{P}^1(\C)\backslash\{0,1,\infty\}\}$ of Calabi-Yau threefolds. Note that the condition ``$g(0)=f(0)=1$'' ensures that the hypergeometric groups $\Gamma(f,g)$ are symplectic for all $14$ pairs of polynomials $f,g$.
 
 The question to determine the arithmeticity or thinness of these $14$ symplectic hypergeometric groups is now completely solved (cf. Table \ref{table:calabiyau}): $7$ of the $14$ groups are arithmetic (by \cite{SS}, \cite{SV}), and other $7$ are thin (by \cite{BT}).
 
 We may now ask the following question:
 \begin{question}\label{mainquestion}
 If we consider the examples where $f=(X-1)^5$ ({{that is}}, the local monodromy is maximally unipotent at $\infty$ and $\alpha=(0,0,0,0,0)$) and $g$ is the product of cyclotomic polynomials such that $g(0)=1$, $g(1)\neq 0$,  and $f,g$ form a primitive pair. In this case also there are precisely $14$ such examples which are determined as follows: if the parameters $(0, 0, 0, 0)$, $(\beta_1, \beta_2, \beta_3, \beta_4)$ correspond to the $14$ symplectic hypergeometric groups, then the parameters $(0, 0, 0, 0, 0)$, $(\beta_1, \beta_2, \beta_3, \beta_4, \frac{1}{2})$ correspond to the $14$ hypergeometric groups $\Gamma(f,g)$ where $f=(X-1)^5$.
 
 Note that, in this case we get $\frac{f(0)}{g(0)}=-1$, therefore each of the $14$ hypergeometric groups $\Gamma(f,g)$ preserves some integral quadratic form $\mathrm{Q}$ on $\Z^5$ and $\Gamma(f,g)\subset\mathrm{O}_\mathrm{Q}(\Z)$ is Zariski dense.
 
 Therefore similar to the $14$ symplectic cases, we may ask the question to determine the arithmeticity or thinness of the $14$ orthogonal cases; it is a special case of Sarnak's question \cite{S}.
 \end{question}
 
 In this article, we show that $2$ of the $14$ orthogonal hypergeometric groups are {\it arithmetic} (cf. Table \ref{table:calabiyau}). In fact, we get the following theorem:
 
 \begin{theorem}\label{maintheorem}
  The hypergeometric groups associated to the pairs of parameters $(0, 0, 0, 0, 0)$, $(\frac{1}{6}, \frac{1}{6}, \frac{5}{6}, \frac{5}{6}, \frac{1}{2})$;  and $(0, 0, 0, 0, 0)$, $(\frac{1}{4}, \frac{1}{4}, \frac{3}{4}, \frac{3}{4}, \frac{1}{2})$,  are arithmetic in the corresponding orthogonal groups.
 \end{theorem}
 
 Note that the symplectic hypergeometric groups associated to the parameters $(0, 0, 0, 0)$, $(\frac{1}{6}, \frac{1}{6}, \frac{5}{6}, \frac{5}{6})$;  and $(0, 0, 0, 0)$, $(\frac{1}{4}, \frac{1}{4}, \frac{3}{4}, \frac{3}{4})$,  are {\it arithmetic} (by \cite{SV} and \cite{SS} resp.); and the orthogonal hypergeometric group associated to the parameter $(0, 0, 0, \frac{1}{2})$, $(\frac{1}{4}, \frac{1}{4}, \frac{3}{4}, \frac{3}{4})$, is {\it thin} (by \cite{EF}).
\subsection*{A table to compare the symplectic and orthogonal cases}
We now list the parameters corresponding to the 14 {\it symplectic} and the 14 {\it orthogonal} hypergeometric groups. In the following list, the parameters $\alpha$ (symplectic) and $\alpha$ (orthogonal) are $(0,0,0,0)$ and $(0,0,0,0,0)$ respectively:
{\renewcommand{\arraystretch}{1.4}   
{\tiny\begin{table}[h]
\addtolength{\tabcolsep}{0pt}
\caption{(cf. \cite[Table 1.1]{SS})}
\newcounter{rownum-3}
\setcounter{rownum-3}{0}
\centering
\begin{tabular}{ | c| c|   c| c| c|}
\hline

  No. & $\beta$ (symplectic) & arithmetic & $\beta$ (orthogonal) & arithmetic\\ \hline
  
 { \refstepcounter{rownum-3}\arabic{rownum-3}\label{arithmeticYYCE-1}} &${\frac{1}{6}}$,${ \frac{1}{6}}$,${ \frac{5}{6}}$,${ \frac{5}{6} }$ & Yes, \cite{SV} &${\frac{1}{6}}$,${ \frac{1}{6}}$,${ \frac{5}{6}}$,${ \frac{5}{6} }$,$\frac{1}{2}$ & Yes\\ \hline

  { \refstepcounter{rownum-3}\arabic{rownum-3}*\label{brav1}} &${ \frac{1}{2}}$,${ \frac{1}{2}}$,${ \frac{1}{2}}$,${ \frac{1}{2}}$ & No, \cite{BT} &${ \frac{1}{2}}$,${ \frac{1}{2}}$,${ \frac{1}{2}}$,${ \frac{1}{2}}$,$\frac{1}{2}$ & ?\\ \hline

 \refstepcounter{rownum-3}\arabic{rownum-3}\label{arithmeticYYCE-4} &$\frac{1}{3}$,$\frac{1}{3}$,$\frac{2}{3}$,$\frac{2}{3}$& Yes, \cite{SS}& $\frac{1}{3}$,$\frac{1}{3}$,$\frac{2}{3}$,$\frac{2}{3}$,$\frac{1}{2}$ &? \\ \hline
  
 {   \refstepcounter{rownum-3}\arabic{rownum-3}*\label{brav2}} &$ { \frac{1}{2}}$,$ { \frac{1}{2}}$,$ { \frac{1}{3}}$,$ { \frac{2}{3}}$ & { No, \cite{BT}} &$ { \frac{1}{2}}$,$ {\frac{1}{2}}$,$ {\frac{1}{3}}$,$ { \frac{2}{3}}$,$\frac{1}{2}$ &?	 \\ \hline
  
  \refstepcounter{rownum-3}\arabic{rownum-3}\label{arithmeticYYCE-5}  &$\frac{1}{4}$,$\frac{1}{4}$,$\frac{3}{4}$,$\frac{3}{4}$ &Yes, \cite{SS} &$\frac{1}{4}$,$\frac{1}{4}$,$\frac{3}{4}$,$\frac{3}{4}$,$\frac{1}{2}$& Yes\\ \hline
  
 {   \refstepcounter{rownum-3}\arabic{rownum-3}*\label{brav3}} &$ { \frac{1}{2}}$,$ { \frac{1}{2}}$,$ { \frac{1}{4}}$,$ { \frac{3}{4}}$  &{ No, \cite{BT}}&$ { \frac{1}{2}}$,$ { \frac{1}{2}}$,$ { \frac{1}{4}}$,$ { \frac{3}{4}}$,$\frac{1}{2}$&  ?\\ \hline
  
\refstepcounter{rownum-3}\arabic{rownum-3}\label{arithmeticYYCE-6}   &$\frac{1}{3}$,$\frac{2}{3}$,$\frac{1}{4}$,$\frac{3}{4}$&Yes, \cite{SS} &$\frac{1}{3}$,$\frac{2}{3}$,$\frac{1}{4}$,$\frac{3}{4}$,$\frac{1}{2}$ &? \\ \hline
  
   {\refstepcounter{rownum-3}\arabic{rownum-3}*\label{brav4}}  &$ { 
\frac{1}{5}}$,$ { \frac{2}{5}}$,$ { \frac{3}{5}}$,$ { \frac{4}{5}}$ &{ No, \cite{BT}}&$ { 
\frac{1}{5}}$,$ { \frac{2}{5}}$,$ { \frac{3}{5}}$,$ { \frac{4}{5}}$,$\frac{1}{2}$ & ?\\ \hline
  
 {\refstepcounter{rownum-3}\arabic{rownum-3}*\label{brav5}}  &$ { \frac{1}{2}}$,$ { \frac{1}{2}}$,$ { \frac{1}{6}}$,$ { \frac{5}{6}}$  & { No, \cite{BT}}&$ { \frac{1}{2}}$,$ { \frac{1}{2}}$,$ { \frac{1}{6}}$,$ { \frac{5}{6}}$,$\frac{1}{2}$&?\\ \hline
  
 \refstepcounter{rownum-3}\arabic{rownum-3}\label{arithmeticYYCE-7} &$\frac{1}{3}$,$\frac{2}{3}$,$\frac{1}{6}$,$\frac{5}{6}$ &Yes,\cite{SS} &$\frac{1}{3}$,$\frac{2}{3}$,$\frac{1}{6}$,$\frac{5}{6}$,$\frac{1}{2}$ &?\\ \hline
  
  {\refstepcounter{rownum-3}\arabic{rownum-3}\label{arithmeticYYCE-2}}  &${ \frac{1}{4}}$,${ \frac{3}{4}}$,${ \frac{1}{6}}$,${ \frac{5}{6}}$ & {Yes, \cite{SV}} &${ \frac{1}{4}}$,${ \frac{3}{4}}$,${ \frac{1}{6}}$,${ \frac{5}{6}}$,$\frac{1}{2}$&? \\ \hline
  
   { \refstepcounter{rownum-3}\arabic{rownum-3}*\label{brav6}} &$ { \frac{1}{8}}$,$ { \frac{3}{8}}$,$ { 
\frac{5}{8}}$,$ { \frac{7}{8}}$  & { No, \cite{BT}}&$ { \frac{1}{8}}$,$ { \frac{3}{8}}$,$ { 
\frac{5}{8}}$,$ { \frac{7}{8}}$,$\frac{1}{2}$& ?\\ \hline
  
  {\refstepcounter{rownum-3}\arabic{rownum-3}\label{arithmeticYYCE-3}} &${ \frac{1}{10}}$,${ \frac{3}{10}}$,${ \frac{7}{10}}$,${ \frac{9}{10}}$  &{ Yes, \cite{SV}}&${ \frac{1}{10}}$,${ \frac{3}{10}}$,${ \frac{7}{10}}$,${ \frac{9}{10}}$,$\frac{1}{2}$&?\\ \hline
  
   { \refstepcounter{rownum-3}\arabic{rownum-3}*\label{brav7}} &$ { \frac{1}{12}}$,$ { \frac{5}{12}}$,$ { \frac{7}{12}}$,$ { \frac{11}{12}}$&{ No, \cite{BT}} &$ { \frac{1}{12}}$,$ { \frac{5}{12}}$,$ { \frac{7}{12}}$,$ { \frac{11}{12}}$,$\frac{1}{2}$&? \\ \hline  
  \end{tabular}
\label{table:calabiyau}
\end{table}}

Therefore 2 of the 14 orthogonal hypergeometric groups are arithmetic, and it will be very interesting to determine the dichotomy (similar to the symplectic cases) in the orthogonal cases also.

We now describe the proof of Theorem \ref{maintheorem}. By using \cite[Theorem 4.5]{BH} (cf. Lemma \ref{signature}), one can compute easily that for all the $14$ orthogonal hypergeometric groups of Table \ref{table:calabiyau}, the quadratic form $\mathrm{Q}$ has signature $(3,2)$ or $(2,3)$, that is, the corresponding orthogonal group $\mathrm{O}_\mathrm{Q}$ has real rank $2$ and $\Q$- rank at least one (by the Hasse-Minkowski theorem). In fact, more is true: all the quadratic forms $\mathrm{Q}$ preserved by the 14 orthogonal hypergeometric groups have {\it two} linearly independent $\mathrm{Q}$- orthogonal isotropic vectors in $\Q^5$ (cf. Section \ref{quadraticforms}), that is, the orthogonal groups $\mathrm{O}_\mathrm{Q}$ have $\Q$- rank two.

For the pairs in Theorem \ref{maintheorem},  we explicitly compute, up to scalar multiples, the quadratic form $\mathrm{Q}$ on $\Q^5$ and get a basis $\{\e_1,\e_2,u,\e_2^*,\e_1^*\}$ of $\Q^5$, satisfying the following: $\mathrm{Q}(\e_i, \e_i*)=1$ for $i=1,2$; $\mathrm{Q}(\e_1, \e_1)=\mathrm{Q}(\e_1^*, \e_1^*)=\mathrm{Q}(\e_2, \e_2)=\mathrm{Q}(\e_2^*, \e_2^*)=\mathrm{Q}(\e_1, \e_2)=\mathrm{Q}(\e_1,\e_2^*)=\mathrm{Q}(\e_2,\e_1^*)=\mathrm{Q}(\e_2^*, \e_1^*)=0$; $\mathrm{Q}(u, u)\neq0$, and $u$ is $\mathrm{Q}$- orthogonal to the vectors $\e_1,\e_2,\e_2^*,\e_1^*$. Note that the existence of such basis for $(\Q^5, \mathrm{Q})$ ensures that the corresponding orthogonal group $\mathrm{O}_\mathrm{Q}$ has $\Q$- rank {\it two}. 

{ Let $\mathrm{SO}_\mathrm{Q}$ be the connected component of the identity element in $\mathrm{O}_\mathrm{Q}(\Q)$.  Then}, with respect to the basis $\{\e_1,\e_2,u,\e_2^*,\e_1^*\}$ of $\Q^5$, the group of diagonal matrices in $\mathrm{SO}_\mathrm{Q}$ form a maximal torus $\mathrm{T}$, the group of upper (resp. lower) triangular matrices in $\mathrm{SO}_\mathrm{Q}$ form a Borel subgroup $\mathrm{B}$ (resp. $\mathrm{B}^-$, opposite to $\mathrm{B}$), and the group of unipotent upper (resp. lower) triangular matrices in $\mathrm{SO}_\mathrm{Q}$ form the unipotent radical $\mathrm{U}$ (resp. $\mathrm{U}^-$, opposed to $\mathrm{U}$) of $\mathrm{B}$ (resp. $\mathrm{B}^-$).

We prove the arithmeticity of the two orthogonal hypergeometric groups of Theorem \ref{maintheorem}, by showing that, with respect to the basis $\{\e_1,\e_2,u,\e_2^*,\e_1^*\}$ of $\Q^5$, $\Gamma(f,g)$ intersects $\mathrm{U}(\Z)$ in a finite index subgroup, that is, $\Gamma(f,g)\cap\mathrm{U}(\Z)$ is a subgroup of finite index in $\mathrm{U}(\Z)$, and use a theorem of Tits \cite{T} (cf. \cite[Theorem 3.5]{Ve}).

Also, to show that $\Gamma(f,g)\cap\mathrm{U}(\Z)$ is a subgroup of finite index in $\mathrm{U}(\Z)$, it is enough to show that $\Gamma(f,g)\cap\mathrm{U}(\Z)$ is Zariski dense in $\mathrm{U}$ (since $\mathrm{U}$ is a nilpotent subgroup of $\GL_5(\R)$;  cf. \cite[Theorem 2.1] {Rag}); and we show it, by showing that, $\Gamma(f,g)\cap\mathrm{U}(\Z)$ contains {\it non-trivial} unipotent elements corresponding to all {\it positive roots} of $\mathrm{SO}_\mathrm{Q}$. 
\begin{rmk}
 It is not very difficult to find the unipotent elements in the two hypergeometric groups of Theorem \ref{maintheorem}, but doesn't seem easy to get them in other cases of Table \ref{table:calabiyau}.
\end{rmk}

\section*{Acknowledgements}
  I am grateful to Professor Slawomir Cynk for introducing me the online Magma Calculator; it was very helpful in finding the isotropic vectors of Table \ref{table:calabiyau2}.  I thank Professor Duco van Straten for his interest in this project and for the discussions. I also thank Professor T. N. Vekataramana for his constant encouragement and support. I am  indebted to Professor Wadim Zudilin for suggesting me Question \ref{mainquestion} and taking interest in this project. I thank Institut f\"ur Mathematik, Johannes Gutenberg-Universit\"at {{and Max Planck Institute for Mathematics}}  for the postdoctoral fellowship{{s}}, and Maple for the computations. {{I also thank the referee for his/her very careful reading of the manuscript and valuable comments.}}
  
\section{Quadratic forms preserved by the orthogonal hypergeometric groups of Table \ref{table:calabiyau}}\label{quadraticforms}
We compute the quadratic forms $\mathrm{Q}$ (unique up to scalar) preserved by $\Gamma(f,g)$, using the method of \cite{V}. Let $f,g$ be a pair of monic polynomials of degree $5$, which are product of cyclotomic polynomials, do not have any common root in $\C$, and form a primitive pair. We also assume that the pair $f,g$ satisfy the condition: $f(0)=-1$ and  $g(0)=1$. Then, $\Gamma(f,g)$ preserves a non-degenerate integral  quadratic form $\mathrm{Q}$ on $\Z^5$ and $\Gamma(f,g)\subset\mathrm{O}_\mathrm{Q}(\Z)$ is Zariski dense \cite[Theorem 6.5]{BH}.

Let $a_1, a_2, a_3 ,a_4, b_1, b_2, b_3, b_4\in\Z$ be the integers, which are defined by:
\[f(X)=X^5+a_4X^4+\cdots+a_1X-1, \  g(X)=X^5+b_4X^4+\cdots+b_1X+1.\]
Let $A$ and $B$ be the companion matrices of $f$ and $g$ respectively. Let $C=A^{-1}B$. Then, one can check easily that
{\tiny\[C=\begin{pmatrix}
1 &0 &0 &0 &-(a_1+b_1)\\
0 &1 &0 &0 &-(a_2+b_2)\\
0 &0 &1 &0 &-(a_3+b_3)\\
0 &0 &0 &1 &-(a_4+b_4)\\
0 &0 &0 &0 &-1\end{pmatrix}.\]}
Let $e_1,e_2,e_3,e_4,e_5$ be the standard basis vectors of $\Q^5$ over $\Q$, and $v$ be the last column vector of $C-\mathrm{I}$, where $\mathrm{I}$ is the identity matrix. Then, we get 
\[v=-(a_1+b_1)e_1-(a_2+b_2)e_2-(a_3+b_3)e_3-(a_4+b_4)e_4-2e_5,\]  \[Ce_5=v+e_5,\] and hence 
\begin{align*}
 Cv&=-(a_1+b_1)e_1-(a_2+b_2)e_2-(a_3+b_3)e_3-(a_4+b_4)e_4-2(v+e_5)\\
 &=v+2e_5-2v-2e_5\\
 &=-v.
 \end{align*}
 Therefore, by using the invariance of $\mathrm{Q}$ under the action of $C$, we get that $v$ is $\mathrm{Q}$- orthogonal to the vectors $e_1, e_2, e_3, e_4$ and $\mathrm{Q}(v,e_5)\neq0$ (since $\mathrm{Q}$ is non-degenerate). We  may now assume that $\mathrm{Q}(v,e_5)=1$.
 
 We now prove a lemma (cf. \cite[Remark 3]{V}):
 \begin{lemma}\label{basis}
  The set $\{v, Av, A^2v, A^3v, A^4v\}$ is linearly independent over $\Q$.
 \end{lemma}
\begin{proof}
 First, we note the following:
 \[Av=A(C-\mathrm{I})e_5=A(A^{-1}B-\mathrm{I})e_5=(B-A)e_5.\] Hence if we {{identify}} the vector space $\Q^5$ as the quotient space $\frac{\Q[X]}{<f(X)>}$, where $<f(X)>$ is the ideal generated by $f(X)$ in the polynomial ring $\Q[X]$, and the action of $A$ on $\Q^5$ as multiplication of $X$ on $\frac{\Q[X]}{<f(X)>}$,  then $Av$ is the vector $f(X)-g(X)$ (mod $<f(X)>$) in $\frac{\Q[X]}{<f(X)>}$. 
 
 Now, suppose on contrary that the set $\{v, Av, A^2v, A^3v, A^4v\}$ is linearly dependent over $\Q$. {{Then,}} the set $\{Av, A^2v, A^3v, A^4v, A^5v\}$ is also linearly dependent over $\Q$, and hence there exists a {\it non-zero} polynomial $h(X)$ of degree $\leq4$ such that \[h(X).(f(X)-g(X))\equiv0 \left(\mbox{mod } f(X)\right) \mbox{ in } \frac{\Q[X]}{<f(X)>}.\] That is, $f(X)$ divides  $h(X).(f(X)-g(X))$. Since $f(X)$ and $g(X)$ are co-prime, $f(X)$ divides $h(X)$ and hence $h(X)=0$ (since the degree ($\leq 4$) of $h(X)$  is less than the degree of $f(X)$), which is a contradiction to our assumption that $h(X)\neq0$. Therefore the set $\{v, Av, A^2v, A^3v, A^4v\}$ is linearly independent over $\Q$.
\end{proof}

By using Lemma \ref{basis}, to determine the quadratic form $\mathrm{Q}$ on $\Q^5$, it is enough to compute $\mathrm{Q}(v, A^jv)$, for $j=0,1,2,3,4$ (since $\mathrm{Q}$ is invariant under the action of $A$, that is, $\mathrm{Q}(A^iv, A^jv)=\mathrm{Q}(A^{i+1}v, A^{j+1}v)$, for any $i,j\in\Z$). Also, since $v$ is $\mathrm{Q}$- orthogonal to the vectors $e_1, e_2, e_3, e_4$ and $\mathrm{Q}(v,e_5)=1$ (say), we get \[\mathrm{Q}(v,A^jv)=\mbox{coefficient of }e_5\mbox{ in }A^jv.\]

We prove further the following lemma, to determine the signature of a quadratic form preserved by the orthogonal hypergeometric groups of Table \ref{table:calabiyau}: 
\begin{lemma}\label{signature}
 The quadratic forms $\mathrm{Q}$ preserved by the $14$ orthogonal hypergeometric groups of Table \ref{table:calabiyau} has signature $(3,2)$ or $(2,3)$; that is, the real rank of the orthogonal group $\mathrm{O}_\mathrm{Q}$ is $2$ and the $\Q$- rank of $\mathrm{O}_\mathrm{Q}$ is $1$ or $2$.
\end{lemma}
\begin{proof}
By \cite[Theorem 4.5]{BH}, the signature $(p,q)$ of the quadratic form $\mathrm{Q}$ preserved by a (orthogonal) hypergeometric group is given by the formula:
\[\left|{p-q}\right|=\left|\sum_{j=1}^n(-1)^{j+m_j}\right|\]
where $m_j=\#\{k: \beta_k<\alpha_j\}$, for all $j=1,2,\ldots,n$.

Note that, for the orthogonal cases of Table \ref{table:calabiyau}, $n=5$, $\alpha_j=0$ and $\beta_j>0$, for all $j=1,2,3,4,5$. Therefore we get $m_j=0$, for all $j=1,2,3,4,5$; and hence 
\[\left|{p-q}\right|=1.\] Also, the non-degeneracy of $\mathrm{Q}$ shows that 
\[p+q=5.\] By using both of the equations in $p,q$, we get that the quadratic forms $\mathrm{Q}$ preserved by the $14$ orthogonal hypergeometric groups of Table \ref{table:calabiyau} has signature $(3,2)$ or $(2,3)$; and hence the real rank of the orthogonal group $\mathrm{O}_\mathrm{Q}$ is $2$ and the $\Q$- rank of $\mathrm{O}_\mathrm{Q}$ is $1$ or $2$ (by the Hasse-Minkowski theorem).
\end{proof}

Since $\mathrm{Q}$ is non-degenerate on $\Q^5$, the existence of two linearly independent isotropic vectors (which are not $\mathrm{Q}$- orthogonal) for $\mathrm{Q}$ on $\Q^5$, is now clear from Lemma \ref{signature}. In Section \ref{proof}, for the hypergeometric groups of Theorem \ref{maintheorem}, we show that there exist two linearly independent isotropic vectors (which are $\mathrm{Q}$-orthogonal) for $\mathrm{Q}$ on $\Q^5$, that is, the $\Q$- rank of $\mathrm{O}_\mathrm{Q}$ is {\it two} for each of these two cases.

In fact, more is true: 
\begin{proposition}
 The orthogonal groups $\mathrm{O}_\mathrm{Q}$, corresponding to the $14$ hypergeometric groups of Table \ref{table:calabiyau}, have $\Q$- rank $2$.
\end{proposition}
\begin{proof}
 For a proof, see Subsection \ref{table:quadraticforms} (cf. Table \ref{table:calabiyau2}).
\end{proof}

 \subsection{A table to list the quadratic forms and their isotropic vectors}\label{table:quadraticforms}
 
 In this subsection we list the matrix forms (with respect to the standard basis $\{e_1,e_2,\ldots,e_5\}$ of $\Q^5$) of the quadratic forms $\mathrm{Q}$ preserved by the $14$ orthogonal hypergeometric groups $\Gamma(f,g)$ of Table \ref{table:calabiyau} (the computations are similar to the Subsection \ref{orthogonal-calabi-yau1}). We also list two linearly independent $\mathrm{Q}$- orthogonal isotropic vectors in $\Q^5$, for each of the quadratic forms $\mathrm{Q}$; it follows that the orthogonal groups $\mathrm{O}_\mathrm{Q}$ have $\Q$- rank {\it two}, for all the orthogonal cases of Table \ref{table:calabiyau}.
 
 Since the companion matrix $A$ (resp. $B$) of $f$ (resp. $g$) maps $e_i$ to $e_{i+1}$ for $1\leq i\leq4$, to know the quadratic forms $\mathrm{Q}$ preserved by the orthogonal hypergeometric groups of Table \ref{table:calabiyau}, it is enough to know the scalars $\mathrm{Q}(e_1,e_j)$ for $1\leq j\leq5$ (since $\mathrm{Q}(e_i,e_j)=\mathrm{Q}(Ae_i,Ae_j)=\mathrm{Q}(e_{i+1},e_{j+1})$ for $1\leq i,j\leq4$).
 
 In Table \ref{table:calabiyau2}, we list the first row \[\left(\mathrm{Q}(e_1,e_1),\mathrm{Q}(e_1,e_2),\mathrm{Q}(e_1,e_3),\mathrm{Q}(e_1,e_4),\mathrm{Q}(e_1,e_5)\right)\] of the matrix form of $\mathrm{Q}$, and two linearly independent $\mathrm{Q}$- orthogonal isotropic vectors in $\Q^5$, for each of the quadratic forms $\mathrm{Q}$.
 
 In the following list, the parameters $\alpha$ and $\beta$ are $(0,0,0,0,0)$ and $(\beta_1,\beta_2,\beta_3,\beta_4,\beta_5)$ resp., that is, 
 \[f(X)=(X-1)^5\quad \mbox{ and }\quad g(X)=\prod_{j=1}^5(X-e^{2\pi i \beta_j}):\]
{\renewcommand{\arraystretch}{1.4}   
{\tiny\begin{table}[h]
\addtolength{\tabcolsep}{-4.5 pt}
\caption{}
\newcounter{rownumb-3}
\setcounter{rownumb-3}{0}
\centering
\begin{tabular}{ | c| c|   c| c| c|}
\hline

  No. & $\beta$ & First row of $\mathrm{Q}$& First isotropic vector & Second isotropic vector\\ \hline
  
 {\refstepcounter{rownumb-3}\arabic{rownumb-3}} &${\frac{1}{6}}$,${ \frac{1}{6}}$,${ \frac{5}{6}}$,${ \frac{5}{6} }$,$\frac{1}{2}$& $(57,39,-7,-57,-71)$&$(1,0,0,1,0)$ & $(-1,-2,5,-7,3)$\\ \hline

  {\refstepcounter{rownumb-3}\arabic{rownumb-3}} &${ \frac{1}{2}}$,${ \frac{1}{2}}$,${ \frac{1}{2}}$,${ \frac{1}{2}}$,$\frac{1}{2}$&$(3,0,-5,0,35)$ &$(18,83,149,129,45)$&$(1,4,7,8,0)$ \\ \hline

 \refstepcounter{rownumb-3}\arabic{rownumb-3} &$\frac{1}{3}$,$\frac{1}{3}$,$\frac{2}{3}$,$\frac{2}{3}$,$\frac{1}{2}$&$(27,5,-37,-27,155)$&$(2,2,3,4,1)$ &$(-1,0,-1,-2,0)$ \\ \hline
  
 {\refstepcounter{rownumb-3}\arabic{rownumb-3}} &$ { \frac{1}{2}}$,$ {\frac{1}{2}}$,$ {\frac{1}{3}}$,$ { \frac{2}{3}}$,$\frac{1}{2}$&$(67,7,-101,-41,547)$ &$(118,365,551,463,111)$&$(1028,2527,4360,4265,0)$	 \\ \hline
  
  \refstepcounter{rownumb-3}\arabic{rownumb-3}  &$\frac{1}{4}$,$\frac{1}{4}$,$\frac{3}{4}$,$\frac{3}{4}$,$\frac{1}{2}$&$(17,7,-15,-25,17)$&$(1,1,1,1,0)$&$(0,1,1,1,1)$\\ \hline
  
 {\refstepcounter{rownumb-3}\arabic{rownumb-3}} &$ { \frac{1}{2}}$,$ { \frac{1}{2}}$,$ { \frac{1}{4}}$,$ { \frac{3}{4}}$,$\frac{1}{2}$ &$(11,3,-13,-13,43)$&$(-1,-2,0,-2,5)$ &$(6,17,15,22,0)$\\ \hline
  
\refstepcounter{rownumb-3}\arabic{rownumb-3}   &$\frac{1}{3}$,$\frac{2}{3}$,$\frac{1}{4}$,$\frac{3}{4}$,$\frac{1}{2}$&$(115,37,-125,-155,307)$&$(32,41,59,45,27)$ &$(211,175,265,117,0)$ \\ \hline
  
   {\refstepcounter{rownumb-3}\arabic{rownumb-3}}  &$ { 
\frac{1}{5}}$,$ { \frac{2}{5}}$,$ { \frac{3}{5}}$,$ { \frac{4}{5}}$,$\frac{1}{2}$&$(89,39,-71,-121,89)$&$(27,74,43,8,68)$ &$(-1793,-1902,3675,-3760,0)$\\ \hline
  
 {\refstepcounter{rownumb-3}\arabic{rownumb-3}}  &$ { \frac{1}{2}}$,$ { \frac{1}{2}}$,$ { \frac{1}{6}}$,$ { \frac{5}{6}}$,$\frac{1}{2}$&$(27,15,-13,-33,-5)$&$(60,103,25,37,119)$&$(-2723,-3423,2247,-3605,0)$\\ \hline
  
 \refstepcounter{rownumb-3}\arabic{rownumb-3} &$\frac{1}{3}$,$\frac{2}{3}$,$\frac{1}{6}$,$\frac{5}{6}$,$\frac{1}{2}$&$(265,151,-119,-329,-119)$ &$(2,1,-1,6,-2)$&$(4,-9,12,-7,0)$\\ \hline
  
  {\refstepcounter{rownumb-3}\arabic{rownumb-3}}  &${ \frac{1}{4}}$,${ \frac{3}{4}}$,${ \frac{1}{6}}$,${ \frac{5}{6}}$,$\frac{1}{2}$&$(35,21,-13,-43,-29)$&$(2,3,1,3,3)$ &$(7,3,-7,21,0)$ \\ \hline
  
   {\refstepcounter{rownumb-3}\arabic{rownumb-3}} &$ { \frac{1}{8}}$,$ { \frac{3}{8}}$,$ { 
\frac{5}{8}}$,$ { \frac{7}{8}}$,$\frac{1}{2}$&$(65,47,1,-49,-63)$&$(4,-3,1,1,1)$ &$(-7,9,-7,1,0)$\\ \hline
  
  {\refstepcounter{rownumb-3}\arabic{rownumb-3}} &${ \frac{1}{10}}$,${ \frac{3}{10}}$,${ \frac{7}{10}}$,${ \frac{9}{10}}$,$\frac{1}{2}$&$(141,115,45,-45,-115)$&$(-4,6,-4,1,-1)$ &$(-7,16,-13,4,0)$\\ \hline
  
   {\refstepcounter{rownumb-3}\arabic{rownumb-3}} &$ { \frac{1}{12}}$,$ { \frac{5}{12}}$,$ { \frac{7}{12}}$,$ { \frac{11}{12}}$,$\frac{1}{2}$& $(257,223,129,-1,-127)$& $(9,9,-17,-2,15)$&$(7,7,-23,17,0)$ \\ \hline  
\end{tabular}
\label{table:calabiyau2}
\end{table}}
 
{{\subsection{Structure of the unipotent groups corresponding to the roots of $\mathrm{SO}_\mathrm{Q}$}\label{structure}
 Let $\{\e_1,\e_2,u,\e_2^*,\e_1^*\}$ be a basis of $\Q^5$ over $\Q$, with respect to which, the matrix form $\mathrm{M}_\mathrm{Q}$ of the quadratic form $\mathrm{Q}$, is:
{\tiny\begin{equation}\label{matrixform}
\mathrm{M}_\mathrm{Q}=\begin{pmatrix}
0 &0 &0 &0 &\lambda_1\\
0 &0 &0 &\lambda_2 &0\\
0 &0 &\lambda_3 &0 &0\\
0 &\lambda_2 &0 &0 &0\\
\lambda_1 &0 &0 &0 &0
\end{pmatrix}\end{equation}}where $\mathrm{Q}(\e_i,\e_i^*)=\lambda_i\in\Q^{*},\ \forall 1\leq{i}\leq 2$, and $\mathrm{Q}(u,u)=\lambda_3\in\Q^{*}$. Then, it can be checked easily that the diagonal matrices in $\mathrm{SO}_\mathrm{Q}$ form a maximal torus $\mathrm{T}$, that is,
{\tiny\[\mathrm{T}=\left\{\begin{pmatrix}
t_1 &0 &0 &0 &0\\
0 &t_2 &0 &0 &0\\
0 &0 &1 &0 &0\\
0 &0 &0 &t_2^{-1} &0\\
0 &0 &0  &0 &t_1^{-1}
\end{pmatrix} : t_i\in\Q^*,\quad \forall\ 1\leq i\leq 2\right\}\]}is a maximal torus in $\mathrm{SO}_\mathrm{Q}$. Once we fix a maximal torus $\mathrm{T}$ in $\mathrm{SO}_\mathrm{Q}$, one may compute the root system $\Phi$ for $\mathrm{SO}_\mathrm{Q}$. If we denote by $\mathbf{t}_i$, the character of $\mathrm{T}$ defined by
{\tiny\[\begin{pmatrix}
t_1 &0 &0 &0 &0\\
0 &t_2 &0 &0 &0\\
0 &0 &1 &0 &0\\
0 &0 &0 &t_2^{-1} &0\\
0 &0 &0  &0 &t_1^{-1}
\end{pmatrix}\longmapsto t_i,\qquad\mbox{for }i=1,2,\]}then the roots are: $$\Phi=\{\mathbf{t}_1, \mathbf{t}_2, \mathbf{t}_1\mathbf{t}_2, \mathbf{t}_1\mathbf{t}_2^{-1},\mathbf{t}_1^{-1},\mathbf{t}_2^{-1}, \mathbf{t}_1^{-1}\mathbf{t}_2^{-1}, \mathbf{t}_2\mathbf{t}_1^{-1}\}.$$ If we fix a set of simple roots $$\Delta=\{ \mathbf{t}_2, \mathbf{t}_1\mathbf{t}_2^{-1}\},$$ then the set of positive roots $$\Phi^+=\{\mathbf{t}_2, \mathbf{t}_1\mathbf{t}_2^{-1}, \mathbf{t}_1, \mathbf{t}_1\mathbf{t}_2\},$$ the set of negative roots $$\Phi^-=\{\mathbf{t}_2^{-1}, \mathbf{t}_2\mathbf{t}_1^{-1}, \mathbf{t}_1^{-1}, \mathbf{t}_1^{-1}\mathbf{t}_2^{-1}\};$$ and $\mathbf{t}_1\mathbf{t}_2$, $\mathbf{t}_1$ are, respectively, the {\it highest}  and  {\it second highest} roots in  $\Phi^+$. 

A computation shows that the unipotent group $\mathrm{U}_\alpha$, corresponding to a positive root $\alpha\in\Phi^+$, is given as follows: 
{\tiny \[\mathrm{U}_{\mathbf{t}_2}=\left\{\begin{pmatrix}
1 &0 &0 &0 &0\\
0 &1 &x &-\frac{\lambda_2}{2\lambda_3}x^2 &0\\
0 &0 &1 &-\frac{\lambda_2}{\lambda_3}x &0\\
0 &0 &0 &1 &0\\
0 &0 &0  &0 &1
\end{pmatrix}: x\in\Q\right\}, \quad \mathrm{U}_{\mathbf{t}_1\mathbf{t}_2^{-1}}=\left\{\begin{pmatrix}
1 &x &0 &0 &0\\
0 &1 &0 &0 &0\\
0 &0 &1 &0 &0\\
0 &0 &0 &1 &-\frac{\lambda_1}{\lambda_2}x\\
0 &0 &0  &0 &1
\end{pmatrix}: x\in\Q\right\},\]}
{\tiny \[\mathrm{U}_{\mathbf{t}_1}=\left\{\begin{pmatrix}
1 &0 &x &0 &-\frac{\lambda_1}{2\lambda_3}x^2\\
0 &1 &0 &0 &0\\
0 &0 &1 &0 &-\frac{\lambda_1}{\lambda_3}x\\
0 &0 &0 &1 &0\\
0 &0 &0  &0 &1
\end{pmatrix}: x\in\Q\right\}, \quad \mathrm{U}_{\mathbf{t}_1\mathbf{t}_2}=\left\{\begin{pmatrix}
1 &0 &0 &x &0\\
0 &1 &0 &0 &-\frac{\lambda_1}{\lambda_2}x\\
0 &0 &1 &0 &0\\
0 &0 &0 &1 &0\\
0 &0 &0  &0 &1
\end{pmatrix}: x\in\Q\right\},\]}where $\lambda_1, \lambda_2,$ and $\lambda_3$ are the (non-zero) anti-diagonal entries of the matrix form $\mathrm{M}_\mathrm{Q}$ of the quadratic form $\mathrm{Q}$ (cf. (\ref{matrixform})).

A similar computation shows that, if $\mathrm{U}_\alpha$ is the unipotent group, corresponding to a positive root $\alpha\in\Phi^+$, then the group $$\mathrm{U}_{\alpha^{-1}}:=\left\{u\left(\frac{\lambda_j}{\lambda_i}\right)^t : u\left(\frac{\lambda_i}{\lambda_j}\right)\in\mathrm{U}_\alpha\right\}$$ is the unipotent group, corresponding to the negative root $\alpha^{-1}\in\Phi^-$, where the element $u\left(\frac{\lambda_j}{\lambda_i}\right)^t$ is defined in the following way: we first replace $\frac{\lambda_i}{\lambda_j}$ with $\frac{\lambda_j}{\lambda_i}$ in the matrix $u\left(\frac{\lambda_i}{\lambda_j}\right)\in\mathrm{U}_\alpha$, and then take the transpose.}}

\section{Proof of Theorem \ref{maintheorem}}\label{proof}
 In this section we explicitly compute  the quadratic forms $\mathrm{Q}$ preserved by the two orthogonal hypergeometric groups of Theorem \ref{maintheorem}, and show that the orthogonal groups $\mathrm{O}_\mathrm{Q}$ have $\Q$- rank two, and prove the arithmeticity of $\Gamma(f,g)$ in $\mathrm{O}_\mathrm{Q}$ by using \cite{T} (cf. \cite[Theorem 3.5]{Ve}).
 \subsection{Arithmeticity of the hypergeometric group associated to the pair $\alpha=(0,0,0,0,0)$, $\beta=(\frac{1}{6},\frac{1}{6},\frac{5}{6},\frac{5}{6},\frac{1}{2})$}\label{orthogonal-calabi-yau1}
This is Example \ref{arithmeticYYCE-1} of Table \ref{table:calabiyau}. In this case \[f(X)=(X-1)^5=X^5-5X^4+10X^3-10X^2+5X-1,\] \[g(X)=(X^2-X+1)^2(X+1)=X^5-X^4+X^3+X^2-X+1.\]

Let $A$ and $B$  be the companion  matrices of $f(X)$ and  $g(X)$ resp., and let $C=A^{-1}B$. Then
{\tiny \[A=\begin{pmatrix}  \begin {array}{rrrrr} 0&0&0&0&1\\ \noalign{\medskip}1&0&0&0&-5
\\ \noalign{\medskip}0&1&0&0&10\\ \noalign{\medskip}0&0&1&0&-10
\\ \noalign{\medskip}0&0&0&1&5\end {array}
  \end{pmatrix},  B=\begin{pmatrix}\begin {array}{rrrrr} 0&0&0&0&-1\\ \noalign{\medskip}1&0&0&0&1
\\ \noalign{\medskip}0&1&0&0&-1\\ \noalign{\medskip}0&0&1&0&-1
\\ \noalign{\medskip}0&0&0&1&1\end {array}
 \end{pmatrix}, C=\begin{pmatrix} \begin {array}{rrrrr} 1&0&0&0&-4\\ \noalign{\medskip}0&1&0&0&9
\\ \noalign{\medskip}0&0&1&0&-11\\ \noalign{\medskip}0&0&0&1&6
\\ \noalign{\medskip}0&0&0&0&-1\end {array}
 \end{pmatrix}.\]}   Let $\Gamma=\Gamma(f,g)=<A,B>$ be the  subgroup  of
$\GL_5(\Z)$ generated by $A$ and $B$. 

\subsection*{The invariant quadratic form} 
Note that $A^{-1}B=B^{-1}A$; and the statements of Lemma \ref{basis} and the succeeding paragraph, are unchanged, if we replace $A$ by $B$. Therefore the vectors in the set $\{v, Bv, B^2v, B^3v, B^4v\}$ form a basis of $\Q^5$ over $\Q$. 

Recall that $v$ is the vector $(C-\mathrm{I})e_5$; and $e_1,e_2,e_3,e_4,e_5$ are the standard basis vectors of $\Q^5$ over $\Q$.
By computation, we get
\[ v=-4e_1+9e_2-11e_3+6e_4-2e_5, Bv=2e_1-6e_2+11e_3-9e_4+4e_5,\]
\[ B^2v=-4e_1+6e_2-10e_3+7e_4-5e_5, B^3v=5e_1-9e_2+11e_3-5e_4+2e_5\]
\[B^4v=-2e_1+7e_2-11e_3+9e_4-3e_5.\]

Recall also from the paragraph preceding to Lemma \ref{basis}, that the vector $v$ is $\mathrm{Q}$- orthogonal to the vectors $e_1,e_2,e_3,e_4$, and $\mathrm{Q}(v, e_5)=1$ (say, since $\mathrm{Q}(v, e_5)\neq0$). Therefore we get \[\mathrm{Q}(v, B^jv)=\mbox{ coefficient of }e_5\mbox{ in }B^jv,\mbox{ for all }j=0,1,2,3,4,\] and hence with respect to the basis $\{v, Bv, B^2v, B^3v, B^4v\}$ of $\Q^5$, the matrix form of $\mathrm{Q}$ is:
 {\tiny \[M_\mathrm{Q}=\begin{pmatrix}
\begin {array}{rrrrr} -2&4&-5&2&-3\\ \noalign{\medskip}4&-2&4&
-5&2\\ \noalign{\medskip}-5&4&-2&4&-5\\ \noalign{\medskip}2&-5&4&-2&4
\\ \noalign{\medskip}-3&2&-5&4&-2\end {array}
\end{pmatrix}.\]} Note that we have used the invariance of $\mathrm{Q}$, under the action of $B$, to write the above matrix form of $\mathrm{Q}$.

If we denote by $T$,  the matrix to change the basis $\{e_1,e_2,e_3,e_4,e_5\}$ to $\{v, Bv, B^2v, B^3v, B^4v\}$, that is,
{\tiny \[T=\begin{pmatrix}
     \begin {array}{rrrrr} -4&2&-4&5&-2\\ \noalign{\medskip}9&-6&6&
-9&7\\ \noalign{\medskip}-11&11&-10&11&-11\\ \noalign{\medskip}6&-9&7&
-5&9\\ \noalign{\medskip}-2&4&-5&2&-3\end {array}
    \end{pmatrix},\]} then the matrix form of $\mathrm{Q}$, with respect to the basis $\{e_1,e_2,e_3,e_4,e_5\}$, is
    \[N_\mathrm{Q}= (T^{-1})^t.M_\mathrm{Q}.T^{-1},\] which is preserved by $A$ and $B$, that is, \[A^t.N_\mathrm{Q}.A=N_\mathrm{Q}=B^t.N_\mathrm{Q}.B,\]
    where $(T^{-1})^t, A^t, B^t$ denote the transpose of the respective matrices.
    
It is now clear from the matrix form $M_\mathrm{Q}$ of $\mathrm{Q}$, that \[\e_1=v+B^3v=e_1+e_4\] is an isotropic vector of $\mathrm{Q}$; and by computation, we get \[\e_1^*=v+B^3v+B^4v=-e_1+7e_2-11e_3+10e_4-3e_5\] is another isotropic vector of $\mathrm{Q}$, such that $\{\e_1,\e_1^*\}$ is linearly independent and $\mathrm{Q}(\e_1, \e_1^*)=1$.

If we denote by $E=\Q\e_1\oplus\Q\e_1^*$, then $\Q^5=E\oplus E^\perp$ and the restriction of $\mathrm{Q}$ on $E^\perp$ is non-degenerate, where $E^\perp$ is the $\mathrm{Q}$- orthogonal complement of $E$ in $\Q^5$. We now, by computation, get a basis $\{u_1,u_2,u_3\}$ of $E^\perp$, where
\[u_1=7v+B^2v+7B^3v+B^4v=e_1+13e_2-21e_3+23e_4-8e_5,\]
\[u_2=4v+3B^3v=-e_1+9e_2-11e_3+9e_4-2e_5,\]
\[u_3=-7v+Bv-B^2v-7B^3v=-e_1-12e_2+21e_3-23e_4+9e_5.\]

If we denote by $L$, the matrix to change the basis $\{e_1,e_2,e_3,e_4,e_5\}$ to $\{\e_1,u_1,u_2,u_3,\e_1^*\}$, that is,
{\tiny \[L=\begin{pmatrix}
     \begin {array}{rrrrr} 1&1&-1&-1&-1\\ \noalign{\medskip}0&13&9&
-12&7\\ \noalign{\medskip}0&-21&-11&21&-11\\ \noalign{\medskip}1&23&9&
-23&10\\ \noalign{\medskip}0&-8&-2&9&-3\end {array}
    \end{pmatrix},\]} then the matrix form of $\mathrm{Q}$, with respect to $\{\e_1,u_1,u_2,u_3,\e_1^*\}$, is:
    {\tiny\[L^t.N_\mathrm{Q}.L=\begin{pmatrix}\begin {array}{rrrrr} 0&0&0&0&1\\ \noalign{\medskip}0&-14&-8&
13&0\\ \noalign{\medskip}0&-8&-2&9&0\\ \noalign{\medskip}0&13&9&-12&0
\\ \noalign{\medskip}1&0&0&0&0\end {array}\end{pmatrix}.\]}

We now change the basis $\{u_1,u_2,u_3\}$ of $E^\perp$ to another basis $\{\e_2,u,\e_2^*\}$ with respect to which, the restriction of $\mathrm{Q}$ on $E^\perp$ is {\it anti-diagonal}, where
\[\e_2=-\frac{1}{2}(3v+B^2v+4B^3v+B^4v)=-e_1-2e_2+5e_3-7e_4+3e_5,\]
\[\e_2^*=Bv+B^4v=e_2+e_5,\]
\[u=(4v+3B^3v)-\e_2-3\e_2^*=8e_2-16e_3+16e_4-8e_5.\]

If we denote by $K$, the matrix to change the basis $\{e_1,e_2,e_3,e_4,e_5\}$ to $\{\e_1,\e_2,u,\e_2^*,\e_1^*\}$, that is,
{\tiny \[K=\begin{pmatrix}
     \begin {array}{rrrrr} 1&-1&0&0&-1\\ \noalign{\medskip}0&-2&8&1
&7\\ \noalign{\medskip}0&5&-16&0&-11\\ \noalign{\medskip}1&-7&16&0&10
\\ \noalign{\medskip}0&3&-8&1&-3\end {array}
    \end{pmatrix},\]} then the matrix form of $\mathrm{Q}$, with respect to the basis $\{\e_1,\e_2,u,\e_2^*,\e_1^*\}$, is:
    {\tiny\[K^t.N_\mathrm{Q}.K=\begin{pmatrix}\begin {array}{rrrrr} 0&0&0&0&1\\ \noalign{\medskip}0&0&0&1&0
\\ \noalign{\medskip}0&0&-8&0&0\\ \noalign{\medskip}0&1&0&0&0
\\ \noalign{\medskip}1&0&0&0&0\end {array}\end{pmatrix}.\]}
Note that, the above matrix form of $\mathrm{Q}$, with respect to the $\Q$- basis $\{\e_1,\e_2,u,\e_2^*,\e_1^*\}$ of $\Q^5$, shows that the $\Q$-rank of  the orthogonal group $\mathrm{O}_\mathrm{Q}$ is {\it two}.

Recall that, with respect to the basis $\{\e_1,\e_2,u,\e_2^*,\e_1^*\}$ of $\Q^5$, the group of diagonal matrices in $\mathrm{SO}_\mathrm{Q}$ form a maximal torus $\mathrm{T}$, the group of upper (resp. lower) triangular matrices in $\mathrm{SO}_\mathrm{Q}$ form a Borel subgroup $\mathrm{B}$ (resp. $\mathrm{B}^-$, opposite to $\mathrm{B}$), and the group of unipotent upper (resp. lower) triangular matrices in $\mathrm{SO}_\mathrm{Q}$ form the unipotent radical $\mathrm{U}$ (resp. $\mathrm{U}^-$, opposed to $\mathrm{U}$) of $\mathrm{B}$ (resp. $\mathrm{B}^-$).

\subsection*{Proof of the arithmeticity of $\Gamma$} 
To prove the arithmeticity of $\Gamma$, we show that with respect to the basis $\{\e_1,\e_2,u,\e_2^*,\e_1^*\}$ of $\Q^5$, $\Gamma$ intersects $\mathrm{U}^-(\Z)$ in a finite index subgroup, that is, $\Gamma\cap\mathrm{U}^-(\Z)$ is a subgroup of finite index in $\mathrm{U}^-(\Z)$, and use a theorem of Tits \cite{T}. Note that the proof also follows from a theorem of Venkataramana \cite[Theorem 3.5]{Ve}.

Also, to show that $\Gamma\cap\mathrm{U}^-(\Z)$ is a subgroup of finite index in $\mathrm{U}^-(\Z)$, it is enough to show that $\Gamma\cap\mathrm{U}^-(\Z)$ is Zariski dense in $\mathrm{U}^-$ (since $\mathrm{U}^-$ is a nilpotent subgroup of $\GL_5(\R)$;  cf. \cite[Theorem 2.1] {Rag}); and we show it, by showing that, $\Gamma\cap\mathrm{U}^-(\Z)$ contains {\it non-trivial} unipotent elements corresponding to all {\it negative} roots of $\mathrm{SO}_\mathrm{Q}$. 

Let $x, y$ denote the matrices $A, B$ resp., with respect to the basis $\{\e_1,\e_2,u,\e_2^*,\e_1^*\}$,  that is, $x=K^{-1}AK$ and $y=K^{-1}BK$. Then 
{\tiny\[x=\begin{pmatrix}
                \begin {array}{rrrrr} 0&-1&0&0&2\\ \noalign{\medskip}0&0&0&0&1
\\ \noalign{\medskip}0&1&-1&0&-1\\ \noalign{\medskip}1&4&-8&2&-4
\\ \noalign{\medskip}0&-4&8&-1&4\end {array}
               \end{pmatrix},\qquad y=\begin{pmatrix}\begin {array}{rrrrr} 0&-4&8&-1&5\\ \noalign{\medskip}0&0&0&0&
1\\ \noalign{\medskip}0&1&-1&0&-1\\ \noalign{\medskip}1&1&0&1&-1
\\ \noalign{\medskip}0&-1&0&0&1\end {array}
\end{pmatrix}.\]}

We now denote by $c_1,c_2,\ldots,c_{19}$, the following words in $x$ and $y$:
\[c_1=x^{-1}y,\quad c_2=y^{-3}c_1y^3,\quad c_3=c_1^{-1}c_2, \quad c_4=y^3c_1y^{-3},\] \[c_5=y^{-9}c_1y^9,\quad c_6=xy^{-1},\quad c_7=y^3c_6y^{-3},\quad c_8=y^6,\quad c_9=c_6c_7^{-1}, \] \[c_{10}=c_9^2,\quad c_{11}=c_{10}c_8,\quad c_{12}=(yc_1c_8c_1^{-1})^3,\quad c_{13}=[c_{12},c_{11}],\] \[ c_{14}=c_{13}c_8^{-64},\quad c_{15}=[c_5,c_{11}],\quad c_{16}=[c_{15},c_{14}],\quad c_{17}=c_{11}^{320}c_{16},\] \[c_{18}=c_{15}^8c_{16},\quad c_{19}=c_{11}^{296}c_{18},\] where $[a,b]$ denotes the commutator of $a, b$, that is, $[a,b]=aba^{-1}b^{-1}$. Then
{\tiny\[c_{11}=\begin{pmatrix}\begin {array}{rrrrr} 1&0&0&0&0\\ \noalign{\medskip}0&1&0&0&0
\\ \noalign{\medskip}0&0&1&0&0\\ \noalign{\medskip}2&0&0&1&0
\\ \noalign{\medskip}0&-2&0&0&1\end {array}\end{pmatrix},\qquad c_{14}=\begin{pmatrix}\begin {array}{rrrrr} 1&0&0&0&0\\ \noalign{\medskip}0&1&0&0&0
\\ \noalign{\medskip}0&8&1&0&0\\ \noalign{\medskip}0&256&64&1&0
\\ \noalign{\medskip}0&0&0&0&1\end {array}
 \end{pmatrix},\]  

\[c_{17}=\begin{pmatrix}\begin {array}{rrrrr} 1&0&0&0&0\\ \noalign{\medskip}0&1&0&0&0
\\ \noalign{\medskip}-16&0&1&0&0\\ \noalign{\medskip}0&0&0&1&0
\\ \noalign{\medskip}1024&0&-128&0&1\end {array}
 \end{pmatrix},\qquad c_{19}=\begin{pmatrix}\begin {array}{rrrrr} 1&0&0&0&0\\ \noalign{\medskip}16&1&0&0&0
\\ \noalign{\medskip}0&0&1&0&0\\ \noalign{\medskip}0&0&0&1&0
\\ \noalign{\medskip}0&0&0&-16&1\end {array} 
 \end{pmatrix}.\]}
 
 It is now clear from the above computations that the elements $c_{11}$, $c_{14}$, $c_{17}$, $c_{19}\in\Gamma\cap\mathrm{U}^-(\Z)$, are non-trivial, and correspond to the {\it negative} { roots} of $\mathrm{SO}_\mathrm{Q}$ {(cf. Subsection \ref{structure})}, that is, $\Gamma\cap\mathrm{U}^-(\Z)$ is Zariski dense in $\mathrm{U}^-$, and hence $\Gamma\cap\mathrm{U}^-(\Z)$ is of finite index in $\mathrm{U}^-(\Z)$ (since $\mathrm{U}^-$ is a nilpotent subgroup of $\GL_5(\R)$). The proof now follows from \cite{T} (cf. \cite[Theorem 3.5]{Ve}). \qed
 
 \subsection{Arithmeticity of the hypergeometric group associated to the pair $\alpha=(0,0,0,0,0)$, $\beta=(\frac{1}{4},\frac{1}{4},\frac{3}{4},\frac{3}{4},\frac{1}{2})$}\label{orthogonal-calabi-yau5}
This is Example \ref{arithmeticYYCE-5} of Table \ref{table:calabiyau}. In this case \[f(X)=(X-1)^5=X^5-5X^4+10X^3-10X^2+5X-1,\] \[g(X)=(X^2+1)^2(X+1)=X^5+X^4+2X^3+2X^2+X+1.\]

Let $A$ and $B$  be the companion  matrices of $f(X)$ and  $g(X)$ resp., and let $C=A^{-1}B$. Then
{\tiny \[A=\begin{pmatrix}  \begin {array}{rrrrr} 0&0&0&0&1\\ \noalign{\medskip}1&0&0&0&-5
\\ \noalign{\medskip}0&1&0&0&10\\ \noalign{\medskip}0&0&1&0&-10
\\ \noalign{\medskip}0&0&0&1&5\end {array}
  \end{pmatrix},  B=\begin{pmatrix}\begin {array}{rrrrr} 0&0&0&0&-1\\ \noalign{\medskip}1&0&0&0&-
1\\ \noalign{\medskip}0&1&0&0&-2\\ \noalign{\medskip}0&0&1&0&-2
\\ \noalign{\medskip}0&0&0&1&-1\end {array}
 \end{pmatrix}, C=\begin{pmatrix} \begin {array}{rrrrr} 1&0&0&0&-6\\ \noalign{\medskip}0&1&0&0&8
\\ \noalign{\medskip}0&0&1&0&-12\\ \noalign{\medskip}0&0&0&1&4
\\ \noalign{\medskip}0&0&0&0&-1\end {array}
 \end{pmatrix}.\]}   Let $\Gamma=\Gamma(f,g)=<A,B>$ be the  subgroup  of
$\GL_5(\Z)$ generated by $A$ and $B$. 

\subsection*{The invariant quadratic form} Let  $e_1,e_2,e_3,e_4,e_5$ be the standard basis vectors of $\Q^5$ over $\Q$, and  $v=(C-\mathrm{I})e_5$. By doing similar computations as in Subsection \ref{orthogonal-calabi-yau1}, we find that the matrix form of $\mathrm{Q}$, with respect to the basis $\{v, Bv, B^2v, B^3v, B^4v\}$ of $\Q^5$,  is:
 {\tiny \[M_\mathrm{Q}=\begin{pmatrix}
\begin {array}{rrrrr} -2&6&-14&14&-2\\ \noalign{\medskip}6&-2&
6&-14&14\\ \noalign{\medskip}-14&6&-2&6&-14\\ \noalign{\medskip}14&-14
&6&-2&6\\ \noalign{\medskip}-2&14&-14&6&-2\end {array}
\end{pmatrix}.\]}

If we denote by $T$,  the matrix to change the basis $\{e_1,e_2,e_3,e_4,e_5\}$ to $\{v, Bv, B^2v, B^3v, B^4v\}$, that is,
{\tiny \[T=\begin{pmatrix}
     \begin {array}{rrrrr} -6&2&-6&14&-14\\ \noalign{\medskip}8&-4&
-4&8&0\\ \noalign{\medskip}-12&12&-16&24&-20\\ \noalign{\medskip}4&-8&0
&12&-4\\ \noalign{\medskip}-2&6&-14&14&-2\end {array}
    \end{pmatrix},\]} then the matrix form of $\mathrm{Q}$, with respect to the basis $\{e_1,e_2,e_3,e_4,e_5\}$, is
    \[N_\mathrm{Q}= (T^{-1})^t.M_\mathrm{Q}.T^{-1},\] which is preserved by $A$ and $B$, that is, \[A^t.N_\mathrm{Q}.A=N_\mathrm{Q}=B^t.N_\mathrm{Q}.B.\]

If we change the basis $\{e_1,e_2,e_3,e_4,e_5\}$ to $\{\e_1,\e_2,u,\e_2^*,\e_1^*\}$ by the matrix $K$, where
{\tiny \[K=\begin{pmatrix}
     \begin {array}{rrrrr} 8&0&-32&-5&{\frac {15}{8}}
\\ \noalign{\medskip}8&4&0&-\frac{5}{4}&{\frac {9}{8}}\\ \noalign{\medskip}8&4
&-32&-{\frac {25}{4}}&{\frac {25}{8}}\\ \noalign{\medskip}8&4&0&-\frac{9}{4}&{
\frac {13}{8}}\\ \noalign{\medskip}0&4&0&-\frac{1}{4}&\frac{7}{4}\end {array}
\end{pmatrix},\]} then the matrix form of $\mathrm{Q}$, with respect to the basis $\{\e_1,\e_2,u,\e_2^*,\e_1^*\}$, is:
    {\tiny\[K^t.N_\mathrm{Q}.K=\begin{pmatrix}\begin {array}{rrrrr} 0&0&0&0&1\\ \noalign{\medskip}0&0&0&1&0
\\ \noalign{\medskip}0&0&-32&0&0\\ \noalign{\medskip}0&1&0&0&0
\\ \noalign{\medskip}1&0&0&0&0\end {array}
\end{pmatrix}.\]}
Note that, the above matrix form of $\mathrm{Q}$, with respect to the $\Q$- basis $\{\e_1,\e_2,u,\e_2^*,\e_1^*\}$ of $\Q^5$, shows that the $\Q$-rank of  the orthogonal group $\mathrm{O}_\mathrm{Q}$ is {\it two}.

\subsection*{Proof of the arithmeticity of $\Gamma$} 
Let $x, y$ denote the matrices $A, B$ resp., with respect to the basis $\{\e_1,\e_2,u,\e_2^*,\e_1^*\}$,  that is, $x=K^{-1}AK$ and $y=K^{-1}BK$. Then 
{\tiny\[x=\begin{pmatrix}
                \begin {array}{rrrrr} 0&-8&-4&-\frac{1}{32}&-{\frac {51}{16}}
\\ \noalign{\medskip}2&20&0&-{\frac {7}{8}}&{\frac {35}{4}}
\\ \noalign{\medskip}0&-4&-1&0&-\frac{7}{4}\\ \noalign{\medskip}0&0&0&0&\frac{1}{2}
\\ \noalign{\medskip}0&-32&0&0&-14\end {array}
               \end{pmatrix},\qquad y=\begin{pmatrix}\begin {array}{rrrrr} 0&-\frac{1}{2}&-4&-\frac{1}{2}&{\frac {3}{32}}
\\ \noalign{\medskip}2&0&0&\frac{3}{8}&0\\ \noalign{\medskip}0&0&-1&-\frac{1}{4}&0
\\ \noalign{\medskip}0&0&0&0&\frac{1}{2}\\ \noalign{\medskip}0&0&0&-2&0
\end {array}
\end{pmatrix}.\]}

We now denote by $c_1,c_2,\ldots,c_{17}$, the following words in $x$ and $y$:
\[c_1=x^{-1}y,\quad c_2=y^{16},\quad c_3=y^{-1}c_1y, \quad c_4=y^{-2}c_1y^{2},\] \[c_5=c_1^{-1}c_4,\quad c_6=y^2c_1y^{-2},\quad c_7=c_4^{-1}c_6,\quad c_8=c_7^4c_2^{-15},\quad c_9=c_1c_8c_1^{-1}, \] \[c_{10}=c_2^{-1}c_9,\quad c_{11}=c_8c_{10}^{-1},\quad c_{12}=xy^{-1},\quad c_{13}=y^{-6}c_{12}y^6,\] \[ c_{14}=c_{13}c_2c_{13}^{-1},\quad c_{15}=[c_{13},c_{14}],\quad c_{16}=[c_{15},c_{10}],\quad c_{17}=c_{16}c_2^{-16384}.\] Then
{\tiny\[c_2=\begin{pmatrix}\begin {array}{rrrrr} 1&0&0&-1&0\\ \noalign{\medskip}0&1&0&0&1
\\ \noalign{\medskip}0&0&1&0&0\\ \noalign{\medskip}0&0&0&1&0
\\ \noalign{\medskip}0&0&0&0&1\end {array}
             \end{pmatrix},\qquad c_{10}=\begin{pmatrix}\begin {array}{rrrrr} 1&0&-128&0&256\\ \noalign{\medskip}0&1&0
&0&0\\ \noalign{\medskip}0&0&1&0&-4\\ \noalign{\medskip}0&0&0&1&0
\\ \noalign{\medskip}0&0&0&0&1\end {array}
 \end{pmatrix},\] \[c_{11}=\begin{pmatrix}\begin {array}{rrrrr} 1&-16&0&0&0\\ \noalign{\medskip}0&1&0&0&0
\\ \noalign{\medskip}0&0&1&0&0\\ \noalign{\medskip}0&0&0&1&16
\\ \noalign{\medskip}0&0&0&0&1\end {array}
\end{pmatrix},\qquad c_{17}=\begin{pmatrix}\begin {array}{rrrrr} 1&0&0&0&0\\ \noalign{\medskip}0&1&8192&
1048576&0\\ \noalign{\medskip}0&0&1&256&0\\ \noalign{\medskip}0&0&0&1&0
\\ \noalign{\medskip}0&0&0&0&1\end {array}
 \end{pmatrix}.\]}
 
 It is now clear from the above computations that the elements $c_2$, $c_{10}$, $c_{11}$, $c_{17}\in\Gamma\cap\mathrm{U}(\Z)$, are non-trivial, and correspond to the {\it positive} { roots} of $\mathrm{SO}_\mathrm{Q}$ {(cf. Subsection \ref{structure})}, that is, $\Gamma\cap\mathrm{U}(\Z)$ is Zariski dense in $\mathrm{U}$, and hence $\Gamma\cap\mathrm{U}(\Z)$ is of finite index in $\mathrm{U}(\Z)$. The proof now follows from \cite{T} (cf. \cite[Theorem 3.5]{Ve}). \qed

\end{document}